\documentclass[a4paper]{amsproc}
\usepackage{amsmath,amsthm,amsfonts,amssymb,amscd,enumerate}
\usepackage{graphicx}
\usepackage{color}
\theoremstyle{plain}
 \newtheorem{theorem}{Theorem}
 \newtheorem{prop}{Proposition}
 \newtheorem{lemma}{Lemma}

\theoremstyle{definition}
 
 \newtheorem{dfn}{Definition}
 \newtheorem{question}{Question}
\theoremstyle{remark}
 
 \numberwithin{equation}{section}

\renewcommand{\le}{\leqslant}
\renewcommand{\ge}{\geqslant}

\title{Relatively Hyperbolic Coxeter Groups with Maximal Flats of codimension 1}

\author{Giang Le}

\address{
Department of Mathematics \\ 
Ohio State University\\ 
Columbus, OH 43210\\
USA} \email{le@math.osu.edu}

\begin{document}

\begin{abstract}
We study relatively hyperbolic Coxeter groups of type $HM$ with maximal Euclidean Coxeter subgroups of codimension 1. Our main result in this paper is that the dimension of these groups is
bounded above.
\end{abstract}

\maketitle
\section{Introduction}
Let $(W,S)$ be a Coxeter system. $W$ is called Coxeter group of {\em type $HM^n$} if $W$ has an
effective, proper and cocompact action on some contractible manifold of
dimension $n$, in this case, $n$ is called the {\em dimension} of $W$. We say a Coxeter group of type $HM$ if it is of type $HM^n$ for
some $n$. Examples of Coxeter groups of type $HM$ are compact reflection groups on
$\mathbb{H}^n$. In \cite{PV} Potyagailo and Vinberg give the bound for the dimension of
right-angled Coxeter groups which have fundamental domain of finite volume and construct examples of those groups up
to dimension 8. From those examples if one truncate the vertices at $\infty$ and
introduces new reflecting faces, corresponding to those truncated vertices, one
can obtain new Coxeter groups of type HM. By Caprace's criterion \cite{Caprace}
these groups are relatively hyperbolic relative to their affine special
subgroups of codimension 1 (if all maximal affine subgroups of a Coxeter group
are of codimesion 1, we say that the group has maximal flats of codimesion 1). A question we can ask about relatively hyperbolic Coxeter groups of type
$HM$ with maximal flats of codimension 1 is that
\begin{question} \label{question:dimension}
Is there an upper bound on the dimension $n$ of relatively hyperbolic Coxeter groups of type $HM^n$ with maximal affine subgroups of codimension 1?
\end{question}
This question is inspired by the example of relatively hyperbolic Coxeter groups
of type $HM$ described above and the result of the paper \cite{PV}. In addition, the boundedness
of dimension is also proven for reflection groups in hyperbolic space which have fundamental domain of finite volume (\cite{Prokhorov}). In this paper we prove some boundedness properties for relatively hyperbolic Coxeter groups
of type $HM$ with maximal flats of codimension 1. In particular, we will prove the two
following theorems:
\begin{theorem}\label{theorem:rightangled_case}
Let $W$ be a right-angled Coxeter group of type $HM$ with maximal Euclidean Coxeter subgroups of codimension 1. Then the dimension of $W$ does not exceed 14.
\end{theorem}
A similar result is also hold for general Coxeter groups of type $HM$.
\begin{theorem}\label{theorem:general_case}
Let $W$ be a relatively hyperbolic Coxeter group of type HM with flats of
codimension 1. Then the dimension of $W$ is less than 996.
\end{theorem}

Notice that a similar result of nonexistence of Gromov Hyperbolic Coxeter groups of type $HM$ is proven by Januszkiewicz and Swiatkowski in \cite{JJ}, which is an extension of a result of E. Vinberg. 

The paper is structured as follows. Section \ref{sec:basic} contains basic definitions and notations. We prove Theorem \ref{theorem:rightangled_case} in section \ref{sec:rightangled_case} and Theorem \ref{theorem:general_case} in section \ref{sec:general_case}. Section \ref{sec:technical} contains some technical points needed for the proofs and the last section contains some examples of relatively hyperbolic Coxeter groups.
\section*{Acknowledgements}
I would like to thank my advisor, Mike Davis, for insightful discussions.

\section{Basic definitions and notations}\label{sec:basic}
\subsection{Coxeter groups}
Recall that a Coxeter group is a group $W$ that has a presentation of
the form:
$$W = \langle s_i \in S, i \in I | (s_is_j)^{m_{ij}} = 1 \rangle $$
where $S$ and $I$ are two sets, $m_{ii}$ = 1 and $m_{ij} \in \{2,3,...,\infty\}$ for $i \ne j$ ($m_{ij} = \infty$ means there is no relation between $s_i$ and $s_j$). $(W,S)$ is called a \emph{Coxeter system}.\\
Given a Coxeter system $(W,S)$ the Caley 2-complex of $W$ can be completed to a
complex $\Sigma$, called the \emph{Davis complex} (see \cite{Davis}). The link $L$
of each vertex in $\Sigma$ is a simplical complex which has exactly $|S|$
vertices, labeled by elements of $S$ and any $k$ vertices $s_1, s_2,....,s_k$
spans a $(k-1)-$simplex if and only if the subgroup generated by $s_1, s_2,....,s_k$ of
$W$ is finite. The matrix $M = (m_{ij})$ is called \emph{Coxeter matrix}. There is a
method, due to Coxeter, of encoding the information in a Coxeter matrix $M$ into
a graph $\Gamma$ with edges labeled by integers greater than 3 or the symbol
$\infty$. This graph is called \emph{Coxeter graph}. The vertex set of $\Gamma$
is $I$. Two vertices $i$ and $j$ are connected by an edge if and only if $m_{ij}
\ge 3$. The edge $\{i,j\}$ is labeled by $m_{ij}$ if $m_{ij} \ge 4$. The graph
$\Gamma$ together with the labeling of its edges is called the {\em Coxeter diagram}
associated to $M$. The diagram is called \emph{positive definite (elliptic),
semidefinite (parabolic),} or \emph{indefinite} if the corresponding matrix $M$
has this property.

The distance $d(u,v)$ between vertices $u$ and $v$ of a diagram $S$ is defined
to be the length (the number of edges) of the shortest path joining $u$ and $v$.
If $u$ and $v$ are not joined by any path, then $d(u,v) = \infty$.

Recall that an indefinite diagram whose proper sub-diagrams are all either
elliptic or parabolic is called a \emph{quasi-Lanner diagram}. The list of all
quasi-Lanner diagrams can be found in \cite{Bourbaki}. It is important for the proof of
Theorem \ref{theorem:general_case} that the distance of any two vertices in a
quasi-Lanner diagram does not exceed 8, and there is at most one pair of
vertices which has distance equal to 8.
\subsection{Coxeter groups of type HM}
Another way to characterize Coxeter group of type $HM$ is via the link $L$ of $W$ or via the quotient of $\Sigma$ by $W$. First, recall that
a space $X$ is a {\em generalized homology n-sphere} if it is a homology n-manifold with the same homology as $S^n$. A pair
$(X,\partial X)$ is a {\em generalized homology $n$-disk} if it is a homology n-manifold with boundary and if it has the same
homology as $(D^n, S^{n-1})$.

When the Coxeter group is type $HM^n$, the link $L(S)$ is a generalized homology $(n-1)$-sphere and the fundamental domain $P$ is a so-called
{\em generalized simple polytope} (or a generalized polytope for short), i.e.
$(P,\partial P)$ is a generalized homology $n$-disk and for each $\sigma_T \in
L$, $(P_T,\partial P_T)$ is a generalized homology disk of dimension $n - |T|$, where $T$ is a subset of the set of generators of $W$ such that the subgroup $W_T$ generated by elements in $T$ is a finite group and $P_T$ is the fundamental domain for $W_T$ (for proof see \cite{Davis}, chapter 10). In this paper, we will use these characteristics of Coxeter group of type $HM$ to prove our theorems. 

\subsection{Some inequalities}
A convex bounded $n$-dimensional polyhedron $P \subset \mathbb{R}^n$ is called
\emph{simple} (respectively, \emph{almost simple}) if the links of its vertices
are simplices (respectively, direct products of simplices). A polyhedron $P$ is
called \emph{edge-simple} if the links of its vertices are simple polyhedrons (or equivalently, the links of its edges are simplices). Note that an almost simple polyhedron is edge-simple and if $P$ is one of the
above three types then its faces are polyhedra of the same type as $P$. Denote
$\alpha_0, \alpha_1,\ldots, \alpha_{n-1}$ the numbers of faces of dimensions $0,
1, \ldots, n-1$ of $P$, respectively, and
$$\alpha_k^{(i)}  = \frac{1}{\alpha_k} \sum_{\substack{\Gamma \subset P \\dim \Gamma = k}}\alpha_i^{\Gamma}$$
denotes the average number of $i$-faces of a $k$-face of $P$ (here
$\alpha_i^{\Gamma}$ is the number of $i$-faces of a face $\Gamma$). The
following theorem is due to Nikulin \cite{Nikulin}.
\begin{theorem} \label{theorem:Nikulin}
For every simple convex bounded polyhedron $P \subset \mathbb{R}^n$ for $i <k
\le [n/2]$ the following estimate holds:
$$\alpha_k^{(i)} < \binom{n-k}{n-i} \frac{\binom{i}{[n/2]} + \binom{i}{[(n+1)/2]}}{\binom{k}{[n/2]} + \binom{k}{[(n+1)/2]}}$$
\end{theorem}
Khovanskii extends this result to edge-simple polyhedron \cite{Khovanskii}.
\begin{theorem} \label{theorem:Khovanskii}
The estimate from theorem \ref{theorem:Nikulin} holds for an edge-simple
polyhedron $P \subset \mathbb{R}^n$.
\end{theorem}

We will use these inequalities for our fundamental domain $P$ of the Coxeter group $W$ of type $HM$. Here $P$ is generally not a polyhedron in $\mathbb{R}^n$, but the links of its vertices are simplices (faces in $L$) and we can still apply these inequalities to $P$ since the proofs of theorems \ref{theorem:Nikulin} and \ref{theorem:Khovanskii} use only combinatorial properties of a polyhedron. 

\section{The cutting process} \label{sec:technical}
In Theorem \ref{theorem:rightangled_case} and Theorem \ref{theorem:general_case} we are considering Coxeter group of
type $HM$ with maximal flats (maximal Euclidean subgroups) of codimension 1. Type $HM$ means that the nerve
$L$ is a generalized homology sphere of dimension $n-1$ and the fundamental
domain $P$ is a generalized simple polytope (of dimension $n$). The condition that all maximal
flats are of codimension 1 guarantees that the nerves of those flats are of
correct dimension (i.e, of dimension $n-2$) and they are also of type $HM$. It is easy to see that for any nerve $L_1$ of an Euclidean subgroup of $W$ of codimension 1, if we cut $L$ along $L_1$, we get two connected components (Alexander duality). We say that a nerve $L_1$ of an Euclidean subgroup of $W$ (of codimension 1) lies in the boundary of $L$ if at least one of those connected components is a cone on $L_1$. To prove Theorem  \ref{theorem:rightangled_case} and Theorem \ref{theorem:general_case} it is crucial that the nerves of all maximal Euclidean subgroups lie in the boundary of $L$.  This can be assumed since if $L_1$ is not in the boundary of $L$ we can cut $L$ along $L_1$ and cone off $L_1$ in each connected  component. Each new piece we get from this cutting process is again a generalized homology sphere (a simple Mayer-Vietoris sequence argument) and has $L_1$ in the boundary.

The opposite process of cutting is gluing. If two nerves $L_1$ and $L_2$ of relatively hyperbolic Coxeter groups of type $HM^n$ with isolated flats of codimension 1 have the same flat, then we can glue $L_1$ and $L_2$ along the flat and get a new nerve of the same type. 

\section{The right-angled case.} \label{sec:rightangled_case}
In this section we prove Theorem \ref{theorem:rightangled_case}. Let $W$ be a
right-angled Coxeter group of type $HM$ with maximal Euclidean subgroups of
codimension 1. Applying the cutting process to the nerve $L$ of $W$, we can assume that the nerves of Euclidean subgroups of codimension 1 of $W$ lie in the boundary of $L$. We follow the proof of the main theorem in Potyagailo-Vinberg's paper \cite{PV}. In
this paper the authors prove the boundedness for right-angled reflection groups
in hyperbolic spaces of finite volume, i.e., the fundamental domain of the
action of a right-angled reflection group on hyperbolic space, which is a polyhedron, has finite volume. Our case is similar to this one
in the sense that the cubes of codimension 1 (these are the fundamental domain
for Euclidean subgroups) can be thought as cusps for the polyhedron. To use their
proof we collapse each cube in the (generalized) polytope $P$ to a point to get another polytope
$P'$. The condition that cubes are isolated in this situation guarantees that
$P'$ is actually a generalized polytope.

We call a vertex in $P'$ a \emph{cusp} if it is the vertex gotten from
collapsing a cube. Two codimension 1 faces of $P'$ are \emph{parallel} if they
meet at a cusp but do not otherwise intersect. We use $F'$ to denote a face of $P'$ gotten from face $F$ of $P$
by collapsing cubes in $P$. Observe that
\begin{enumerate}
\item $P'$ is not a simple generalized polytope as $P$, but it is an edge-simple generalized polytope because the links of edges in $P'$ are the same as the links of the same edges in $P$. And we can still apply the Nikulin-Khovanskii inequality for $P'$.
\item Since $P$ is simple, any face of $P$ is also simple. And if a face $F$ has
non-empty intersection with a cube of codimension 1 of $P$ then the intersection is a
a cube of codimension 1 of $F$ (so the above definition of cusp is still valid
for faces of $P'$). 
\end{enumerate}

\begin{prop}
Let $F_1', F_2',...$ be faces of codimension 1 of $P'$. Then
\begin{enumerate}
\item \label{first} If $F_1', F_2', F_3'$ are pairwise mutually adjacent, then they meet at a
$(n-3)$-dimensional face;
\item \label{second} if $F_1'$ and $F_2'$ are parallel and $F_3'$ is adjacent to them ($F_1'$ and $F_2'$  are not adjacent), then $F_1', F_2',
F_3'$ meet at a cusp;
\item \label{third} if $F_1'$ and $F_2'$ are parallel and $F_3'$ and $F_4'$ are adjacent to them, then $F_1', F_2',
F_3', F_4'$ meet at a cusp.
\end{enumerate}
\end{prop}
\begin{proof}
Part \ref{first} follows from the property of the nerve $L$ that $L$ is a {\em flag
complex} (recall that a flag complex is a simplicial complex in which any finite collection of vertices is a simplex if there are edges connecting any two vertices from the collection, and since $W$ is right-angled Coxeter group the nerve $L$ is a flag complex). If $F_1', F_2', F_3'$ are pairwise mutually adjacent, then the same is true for $F_1, F_2, F_3$. Then the vertices $v_1, v_2, v_3$ in the nerve $L$ corresponding to $F_1, F_2, F_3$ are connected by edges. Since $L$ is a flag complex, $v_1, v_2, v_3$ form a triangle in $L$. This means $F_1, F_2, F_3$ meet at a $(n-3)-$dimensional face.

Part \ref{second}: 4 vertices $v_1, v_2, v_3, v_4$ in the nerve $L$ corresponding to $F_1, F_2, F_3$ and the cusp form a square. If this square is an empty square, it must be part of an octahedron face in $L$ (dual of the cube face of $P$) because of the isolated cube condition of $P$. In this case all the faces $F_1, F_2, F_3$ have intersection with the same cube in $P$, so they meet at a cusp. If the square is not empty, there must be an edge connecting two opposite vertices.  And it's not in the case of part \ref{first} (no edge between $v_1$ and $v_2$) then there must be an edge connecting $v_3$ and $v_4$.

Part \ref{third}: Let $v_1, v_2, v_3, v_4$ are vertices in $L$ corresponding to $F_1, F_2, F_3, F_4$, and $u_1, u_2$ corresponding to the two cubes (cusps) which $F_1', F_2', F_3'$ and $F_1', F_2', F_3'$ meet following part \ref{second}. If $u_1$ and $u_2$ are the same, then nothing needs to prove. If not, then $u_1, u_2$ are cone points over the nerves $L_1, L_2$ of Euclidean subgroups of $W$. Now $L_1$ and $L_2$ have two common vertices $v_1, v_2$ which are not connected by an edge, they must be the same nerve of an Euclidean subgroup of $W$ by the criterion of isolated flat for relatively hyperbolic Coxter group (\cite{Caprace}). Then in this case $L$ is the suspension of $L_1$ and thus $W$ has dimension $n-1$ which is a contradiction.
\end{proof}

Now we will consider faces of dimensions 2, 3, 4 and 5 of $P'$. Denote $a_k =
a_k(P')$ the number of $k$-faces of $P'$. In particular, $a_0(P')$ is the total
number of ordinary vertices and cusps of $P'$. The number of cusps will be
denoted by $c$. For each case below, the number of faces is for the
corresponding face of $P'$.

Faces of dimension 2: Let $F'$ be any face of dimension 2 of $P'$. We will prove
that $a_1(F') + c(F') \ge 5$. Since $P$ is a right-angled polytope, the two
dimensional polyhedron $F$ of $P$ has at least 5 edges. If $F'$ does not contain
any cusps, then $F = F'$, and the inequality is true. If $F'$ has cusps, it
means some edges of $F$ are also edges of squares in cubes, those edges will be
collapsed to cusps in $F'$, those edges of $F$ are not edges in cubes will be
remained edges in $F'$, so the inequality is still true for this case.

The difference $a_1 + c - 5$ will be called the \emph{excess} of $ F'$ and
denoted by $e = ex(F')$.

Faces of dimension 3: For each 2-dimensional face $Q'$ of some 3-dimensional
face $F'$ of $P'$ we have
\begin{equation} \label{2face1}
a_1(Q') + c(Q') = 5 + ex(Q')
\end{equation}
Summing over all $Q'$ and taking into account that each edge of $F'$ belongs to
2 faces and each cusp belongs to 4 faces, we get
\begin{equation} \label{2face2}
2a_1 + 4c = 5a_2 + \sum_{Q'}ex(Q')
\end{equation}
On the other hand, eliminating $a_0$ from the Euler equation $a_0 - a_1 + a_2 =
2$ and the obvious equation $2a_1 = 3a_0 + c$ (counting the number of vertices
in edges in two ways, counting by edges, counting by vertices - since $F$ is
simple each original vertex belongs to 3 edges) gives
\begin{equation} \label{3face1}
a_1 + c = 3a_2 - 6
\end{equation}
Substituting this into \eqref{2face2}, we finally obtain
\begin{equation} \label{3face2}
a_2 + 2c = 12 + \sum_{Q'} ex(Q') \ge 12.
\end{equation}
Want to prove the inequality
\begin{equation} \label{3face3}
a_2 \ge 6
\end{equation}
Here the argument from \cite{PV} can be used because for the 3 dimension case
$P$ can be thought of hyperbolic polytope, this follows from Steinitz' theorem (Theorem 4.1 in \cite{Ziegler}).

It follows from \eqref{3face1} and \eqref{3face3} that
\begin{equation} \label{3face4}
a_2 + c \ge 9
\end{equation}

Faces of dimension 4: Let $Q'$ be some 3 dimensional face of 4-dimensional face $F'$ of
$P'$. There are $a_2(Q')$ 3-dimensional faces adjacent to $Q'$ and, for each
cusp of $Q'$, there is an extra 3-dimensional face having only this cusp in
common with $Q'$. Together with $Q'$, this gives at least $1 + a_2(Q') + c(Q')$
3-dimensional faces of $F'$. So \eqref{3face4} implies
\begin{equation} \label{4face1}
a_3 \ge 10.
\end{equation}
We need a more subtle inequality
\begin{equation} \label{4face2}
a_3 + c \ge 15.
\end{equation}
To prove it, take again any hyperface $Q'$ of $F'$. There are at least $1 +
a_2(Q') + c(Q')$ hyperfaces meeting $Q'$ and at least $c(Q')$ cusps, so $a_3 +
c\ge 1 + a_2(F') + 2c(F')$. If $a_2(F') + 2c(F') \ge 14$, then \eqref{4face2}
follows. By \eqref{3face2} we have $a_2(Q') + 2c(Q') \ge 12$. Consider two
cases.

Let $a_2(Q') + 2c(Q') = 13$. Then \eqref{3face2} implies that all but one
2-faces of $Q'$ have zero excess. Let $q'$ be a 2-face of $Q'$ with zero excess,
i.e. $a_1(q') + c(q') = 5$. Since $c(q') \le 2$, we have
\begin{equation} \label{4face3}
1 + a_1(q') + 2c(q') \le 8.
\end{equation}
Let $Q_1'$ be the hyperface of $P'$ adjacent to $Q'$ along $q'$. By
\eqref{3face4} we have $a_2(Q_1') + c(Q_1') \ge 9$. Comparing this with
\eqref{4face3}, we have
$$(a_2(Q_1') - 1 - a_1(q') - c(q')) + (c(Q_1') - c(q')) \ge 1$$
we see that $Q_1'$ must have either a 2-face $q_1'$ not intersecting $Q'$, or a
cusp beyond $Q'$ (2-face of $Q_1'$ if intersects $Q'$ then it has common edge
(so $a_1(q')$), 1 is contributed to $q'$, $c(q')$ is contributed to the number
of 2-faces of $Q_1'$ which are parallel to $q'$). In the first case the
hyperface adjacent to $Q_1'$ along $q_1'$ does not intersect $Q'$ by proposition
\ref{first}, \ref{second}. So in both case \eqref{4face2} holds.

Let $a_2(Q') + 2c(Q') = 12$. Then \eqref{3face2} implies that all 2-faces of
$Q'$ have zero excess. Let $q'$ be any of them. Then $q'$ is a triangle with two
cusps, or a quadrilateral with one cusp, or else a pentagon without cusps. If
$q'$ is not a triangle, then
\begin{equation} \label{4face4}
1 + a_1(q') + 2c(q') \le 7.
\end{equation}
Consider the hyperface $Q_1'$ of $P'$ adjacent to $Q'$ along $q'$. Then
\eqref{4face4} implies that $Q_1'$ has at least two 2-faces not intersecting
$Q'$ or cusps beyond $Q'$, whence again \eqref{4face2} follows.

Let finally all 2-faces of $Q'$ be triangles with two cusps. Take any parallel
2-faces $q_1'$ and $q_2'$ of $Q'$, and let $Q_1'$ and $Q_2'$ be the hyperfaces
of $P'$ adjacent to $Q'$ along $q_1'$ and $q_2'$ respectively. By the above each
of them must have either a 2-face not intersecting $Q'$ or a cusp beyond $Q$. If
these are two 2-faces, then the hyperfaces of $P'$ adjacent to $Q_1'$ and $Q_2'$
along them, cannot coincide by proposition \ref{third}. If these are two cusps,
then they cannot coincide as $Q_1'$ and $Q_2'$ are parallel at a cusp of $Q'$.
So in all the cases \eqref{4face2} holds.

Faces of dimension 5: Take any hyperface $Q'$ of 5-dimensional face $F'$ of $P'$. There
are $a_3(Q')$ hyperfaces adjacent to $Q'$ and, for each cusp of $Q'$, there is
an extra hyperface having only this cusp in common with $Q'$. Together with
$Q'$, this gives at least $1 + a_3(Q') + c(Q')$ hyperfaces. So \eqref{4face2}
implies
\begin{equation} \label{5face1}
a_4 \ge 16.
\end{equation}
Using the Nikulin-Khovanskii inequality (theorems \ref{theorem:Nikulin},
\ref{theorem:Khovanskii}) gives for the average number $a^4_5$ of 4-faces of a
5-face of $P'$

$$
a^4_5 < 
\begin{cases} 
\frac{10(n-4)}{n-8}, &\text{if $n$ is even,} \\
\frac{10(n-3)}{n-7}, & \text{if $n$ is odd.} 
\end{cases}
$$
On the other hand, it follows from \eqref{5face1} that $a^4_5 \ge 16$. In both
case this means that $n \le 14$.

\section{The general case.} \label{sec:general_case}
In the following sections we will prove Theorem~\ref{theorem:general_case} using
the argument in Prokhorov's paper \cite{Prokhorov}. Applying the cutting process in section 3, we can assume that the nerve $L$ of the Coxeter group $W$ has all nerves of (maximal) Euclidean subgroups of codimension one on its boundary.

Let recall the notations we use in this section. Let $W$ be a Coxeter group and $L$ is its nerve which has the type as
described above, $P'$ is the fundamental domain of $W$ action and it is the dual
of $L$.  $P$ is gotten from $P$ by collapsing all the $(n-1)$-faces which are corresponding to the cone points one the nerves of Euclidean subgroups of codimension 1 to points. All dihedral angles
of $P$ are of the form $\pi/k$, $k = 2, 3, ..., \infty$. There are two types of
vertices in $P$: {\em finite vertices} and {\em infinite vertices}. Finite vertices are actual
vertices in $P$ (and $P'$), infinite vertices are the ones which corresponding with those
Euclidean subgroups. The link of a finite
vertex is a simplex (this is the nerve of a finite subgroup and is stabilizer of
the vertex), the link of infinite vertex is a product of simplexes
(and it is the dual of the nerve of Euclidean subgroup, which is a joint of simplices). Since the links of vertices of $P$ are either simplices or products of simplices, $P$ is almost simple.

A 3 dimensional face $F$ of $P$ is called \emph{bad} if it has the type of a
triangular bi-pyramid (see Figure 1). In this polyhedron $F$, vertices $v_2,
v_3,$ and $v_4$ are at infinity because their links are squares which are
products of simplices. These squares are the nerves of Euclidean subgroups so
they are Euclidean squares. Thus, dihedral angles of $F$ at vertices $v_2, v_3,
v_4$ are equal to $\pi/2$. Then, the links of vertices $v_1$ and $v_5$ are
triangles with angles equal to $\pi/2$, this means they are spherical triangles
and so vertices $v_1$ and $v_5$ are finite vertices.

\begin{figure}
  \centering
     \def\svgwidth{\columnwidth}
    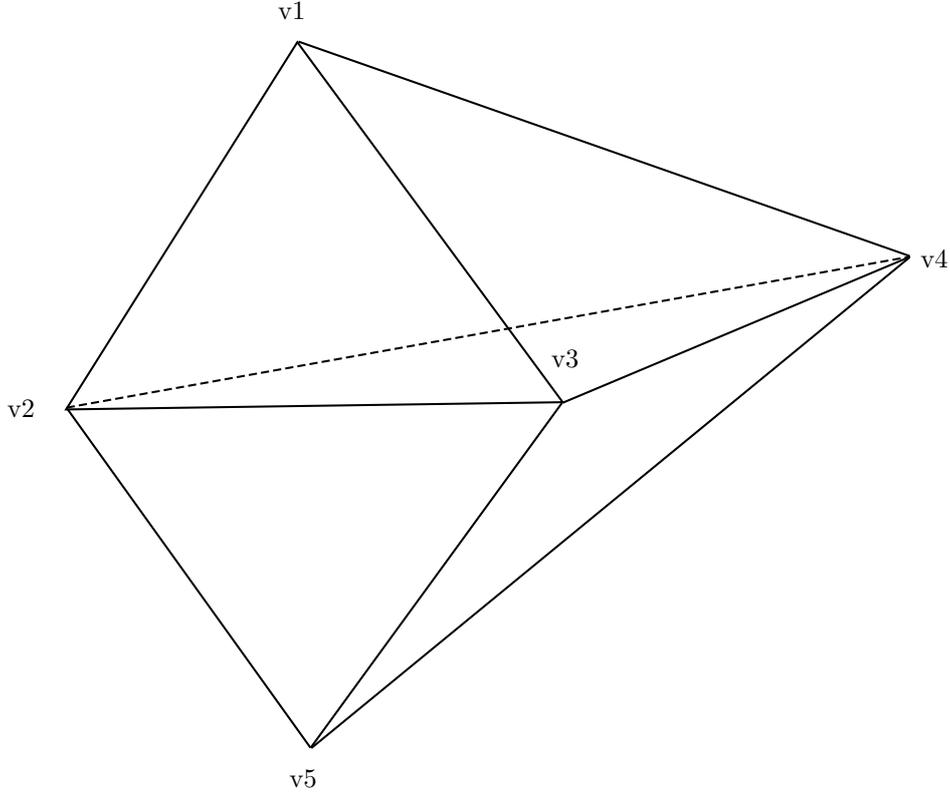
    \caption{ Bad 3-dimensional face}
\end{figure}

3-faces of $P$ different from the bad ones are called \emph{good}. The following
lemmas are from \cite{Prokhorov}. Although they are stated for polyhedron of
finite volume for a reflection group in hyperbolic space, their proof uses only
combinatorial properties of those polyhedra, which are satisfied in our case.

\begin{lemma} \label{lemma:badgood3face}
Suppose that $\alpha_3$ is the number of all 3-faces of $P$, $\alpha'_3$ is the
number of bad 3-faces, and $\alpha''_3 = \alpha_3 - \alpha'_3$ is the number of
good 3-faces of $P$. Then for $n \ge 8$
$$\frac{\alpha'_3}{\alpha_3} = p < \begin{cases} (3n+6)/4(n-3), &\text{$n$ even} \\ (3n+9)/4(n-2), &\text{$n$ odd} \end{cases} $$
\end{lemma}

The following definition is needed for the next lemma:
\begin{dfn}
A dihedral angle $(F,\{\Gamma_1, \Gamma_2\})$ of a 3 dimensional face $F$ of $P$
is defined to be a pair consisting of a 3-face $F$ and a set of two 2-faces
$\Gamma_1$ and $\Gamma_2$ of $F$ which intersect in an edge.
\end{dfn}

\begin{lemma} \label{lemma:inequality}
Let $P$ be an n-dimensional almost simple polyhedron, and let $C$ be a positive
number. Suppose that dihedral angles (edges) of 3-faces of $P$ can be assigned
nonnegative weighs so that:
\begin{enumerate}
\item \label{condition1} the sum $\sigma(\Gamma^1)$ of weights over 3-faces at every edge
$\Gamma^1$ does not exceed C(n-1), and
\item \label{condition2} the sum $\sigma(\Gamma^3)$ of weights over all edges of any good 3-face
$\Gamma^3$ is not less than $7 - k$, where k is the number of 2-faces of
$\Gamma^3$.
\end{enumerate}
Then $n < 96C + 68$.
\end{lemma}

\subsection{Proof of Theorem \ref{theorem:general_case}}
To prove Theorem \ref{theorem:general_case} we will use lemma
\ref{lemma:inequality}. In this section, we will assign dihedral angles of
3-faces of $P$ with weights so that the conditions of lemma
\ref{lemma:inequality} hold for $C = 29/3$. By lemma \ref{lemma:inequality}: $n < 96 \times 29/3 +
66 = 996$, which is the statement of theorem \ref{theorem:general_case}.
\\

With every unordered set of 2-faces $\{ \Gamma_1, \Gamma_2,...,\Gamma_s\} = K$
of a fixed 3-face $F$ of a fundamental polyhedron $P$ we associate a sub-diagram
$S(F,K)$ of the diagram $S$ of $P$ which is generated by the vertices
corresponding to $(n-1)$-faces of $P$ which either contain $F$ or intersect $F$
in a 2-face from K. The vertices of $S(F,K)$ corresponding to the faces of the
second type are called \emph{marked}. A diagram of a dihedral angle
$(F,\{\Gamma_1, \Gamma_2\})$ of a 3-face $F$ of $D$ is defined to be the
sub-diagram $S(F,K)$, where $K = \{ \Gamma_1, \Gamma_2\}$. This is an elliptic
sub-diagram of rank $n-1$, since it is diagram of the stabilizer of the common
edge of $\Gamma_1$ and $\Gamma_2$.

To each dihedral angle $(F,\{\Gamma_1, \Gamma_2\})$ of $P$ assign weight 1 if
the distance between marked vertices in its diagram is at most 7, 1/3 if the
inequality $7 < d(u,v) \le 15$ holds for marked vertices $u$ and $v$ in
$(F,\{\Gamma_1, \Gamma_2\})$, and 0 otherwise.

The following proposition is a corollary from the classification of Euclidean
Coxeter diagrams (see, for example, \cite{Vinberg}).
\begin{prop} \label{prop:EuclideanCoxeter}
For a connected Euclidean Coxeter diagram $T$ of rank $n$ the number of
unordered pairs of vertices with distance at most $C$ between them does not
exceed $Cn$.
\end{prop}

A \emph{complete diagram} of a 3-face $F$ is defined to be the diagram $S(F,K)$,
where $K$ is the set of all 2-faces of $F$. Among all possible sub-diagrams
$S(F,K') \subset S(F,K)$, where $K' \subset K$, consider the following types

\begin{figure}
   \centering
   \def\svgwidth{\columnwidth}
   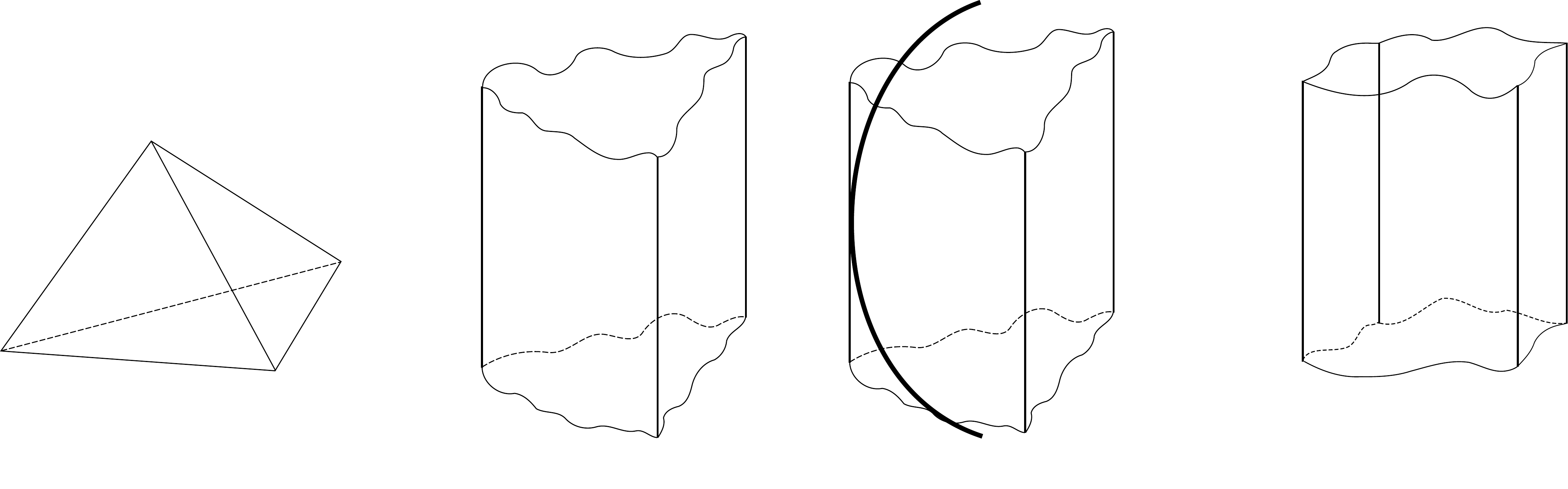
   \caption{}
\end{figure}

\begin{itemize}
\item \emph{Type 1:} $K'$ contains four 2-faces. Every indefinite sub-diagram of
$S(F,K')$ contains all marked vertices (since any sub-diagram of $S(F,K')$
containing at most 3 marked vertices lies in stabilizer of one of the vertices,
thus it is finite if the vertex is finite or it is Euclidean if the vertex is
infinite). A diagram of a dihedral angle of $F$ is obtained from $S(F,K')$
excluded two marked vertices.
\item \emph{Type 2:} $K'$ contains three 2-faces, any two of which intersect at an edge. The diagram $S(F,K')$ is
indefinite (since it does not fix any vertex, it cannot be finite or Euclidean).
A diagram of a dihedral angle of $F$ is obtained from $S(F,K')$ excluded one
marked vertex.
\item \emph{Type 3:} $K'$ contains three 2-faces, two of which intersect at a vertex
at infinite, the other two pair of 2-faces intersect at edges. The diagram
$S(F,K')$ is indefinite (since it does not fix any vertex, it cannot be finite
or Euclidean). A diagram of a dihedral angle of $F$ is obtained from $S(F,K')$
excluded one from two marked vertices corresponding to two 2-face intersecting
at infinity. If the third marked vertex is excluded from $S(F,K')$, we obtain a
diagram of a dihedral angle of $F$.
\item \emph{Type 4:} $K'$ contains four 2-faces which can be divided into two groups
$K'_1$ and $K'_2$. Each of $K'_1$ and $K'_2$ contains two opposite 2-faces. It
is clear that $S(F,K'_1)$ and $S(F,K'_2)$ are indefinite sub-diagrams, and any
indefinite sub-diagram of $S(F,K')$ must contain both marked vertices from some
pair. A diagram of a dihedral angle of $F$ is obtained from $S(F,K')$ excluded
two marked vertices from different pairs $K'_1$ and $K'_2$.
\end{itemize}
The sum $\sigma(F)$ of weights on a face $F$ is defined to be the sum of weights
of all sub-digrams of dihedral angles of $F$ contained in $S(F,K)$. Similarly, we
can define the sum of weights on a sub-diagram $S(F,K') \subset S(F,K)$.

\begin{lemma} \label{lemma:weightof3face}
The following statements are true:
\begin{enumerate}
\item \label{type1} If a complete diagram of a face $F$ contains a sub-diagram $S(F,K')$ of
type 1, then $\sigma(F) \ge 3$.
\item \label{type2} If a complete diagram of a face $F$ contains a sub-diagram $S(F,K')$ of
type 2, then $\sigma(F) \ge 2$.
\item \label{type3} The sum of weights on sub-diagram of type 3 is not less than 1.
\item \label{type4}The sum of weights on sub-diagram of type 4 is not less than 1/3.
\end{enumerate}
\end{lemma}
\begin{proof}
for \ref{type1}, \ref{type2}, \ref{type3}: any indefinite sub-diagram of
$S(F,K')$ contains all marked vertices. Consider the minimum indefinite
sub-diagram $M \subset S(F,K')$. If a non-marked vertex is excluded from $M$, we
cannot get an infinite diagram since $M$ is minimal. Thus we get a finite or
Euclidean sub-diagram. Obviously, if any marked vertex is excluded from $M$, one
obtains a finite or an Euclidean sub-diagram. That means $M$ is a quasi-Lanner
diagram. And since $M$ is minimal, it is connected.

Notice that since $M$ is connected graph, there are at least three pair of
marked vertices in $S(F,K')$ of type 1 which are joined in $M$ by a path
containing no other marked vertices from $S(F,K')$. Thus, by the classification
of quasi-Lanner diagram (see \cite{Bourbaki}) for all such pairs $u, v$ we have $d(u,v) \le
7$. In case \ref{type2} there exist two pairs satisfying this property, and in
case \ref{type3} at least one such pair exists.

For case \ref{type4}, by similar argument as above we have each pair $K'_1,
K'_2$ of marked vertices is contained in a quasi-Lanner diagram which has no
marked vertices from another pair. Let $M$ and $N$ be such quasi-Lanner
diagrams. If $M$ and $N$ are orthogonal then four marked vertices from type 4
would form subgroup $D_{\infty} \times D_{\infty}$ which is Euclidean. This
means the 4 2-faces of $F$ in type 4 meet at vertex at infinity. But this is not
the case we consider in type 4. Thus, $M$ and $N$ are joined by an edge or they
have common vertex. Therefore there exist marked vertices $u$ and $v$ from
different pairs which can be joined by a path not containing other mark
vertices. By the classification of quasi-Lanner diagram we see that $d(u,v) \le
15$. From the definition of weights of dihedral angle the sum of weights on
sub-diagram of type 4 is not less than 1/3.
\end{proof}
We finish the proof of the theorem \ref{theorem:general_case} by proving the
following statement:
\begin{prop} \label{prop:checkcondition}
The fundamental domain $P$ satisfies the conditions of lemma
\ref{lemma:inequality} for $C = 29/3$.
\end{prop}
\begin{proof}
We need to check conditions \ref{condition1} and \ref{condition2} from lemma
\ref{lemma:inequality}. The diagram of an edge $\Gamma^1$ of $P$ is elliptic of
order $n-1$ (an edge in $P$ is an intersection of $n-1$ codimention 1 faces of
$P$). The diagram of a dihedral angle $S(F,\{\Gamma_1, \Gamma_2\})$ of a 3-face
$F$ such that $\Gamma_1 \cap \Gamma_2 = \Gamma^1$ is completely determined by
two marked vertices in the diagram of $\Gamma^1$. By proposition
\ref{prop:EuclideanCoxeter} the number of pairs of vertices in the diagram of
the edge which are at distance at most 15 (respectively, at most 7) cannot
exceed $15(n-1)$ (respectively, $7(n-1)$).  We see that the sum of wights at the
edge is at most $7(n-1) + 15(n-1)/3 = 29(n-1)/3$, and so condition
\ref{condition1} of lemma \ref{lemma:inequality} holds.

It is sufficient to verify condition \ref{condition2} of lemma
\ref{lemma:inequality} for 3-faces which contain no more than six 2-faces. It is not difficult to find all of them (it was done in Prokhorov's paper \cite{Prokhorov}). Figure 3 includes complete list,
except a cone over a pentagon, which is not a direct product of simplexes and so
can not be a 3-face for $P$.

\begin{figure}
  \centering
   \def\svgwidth{\columnwidth}
    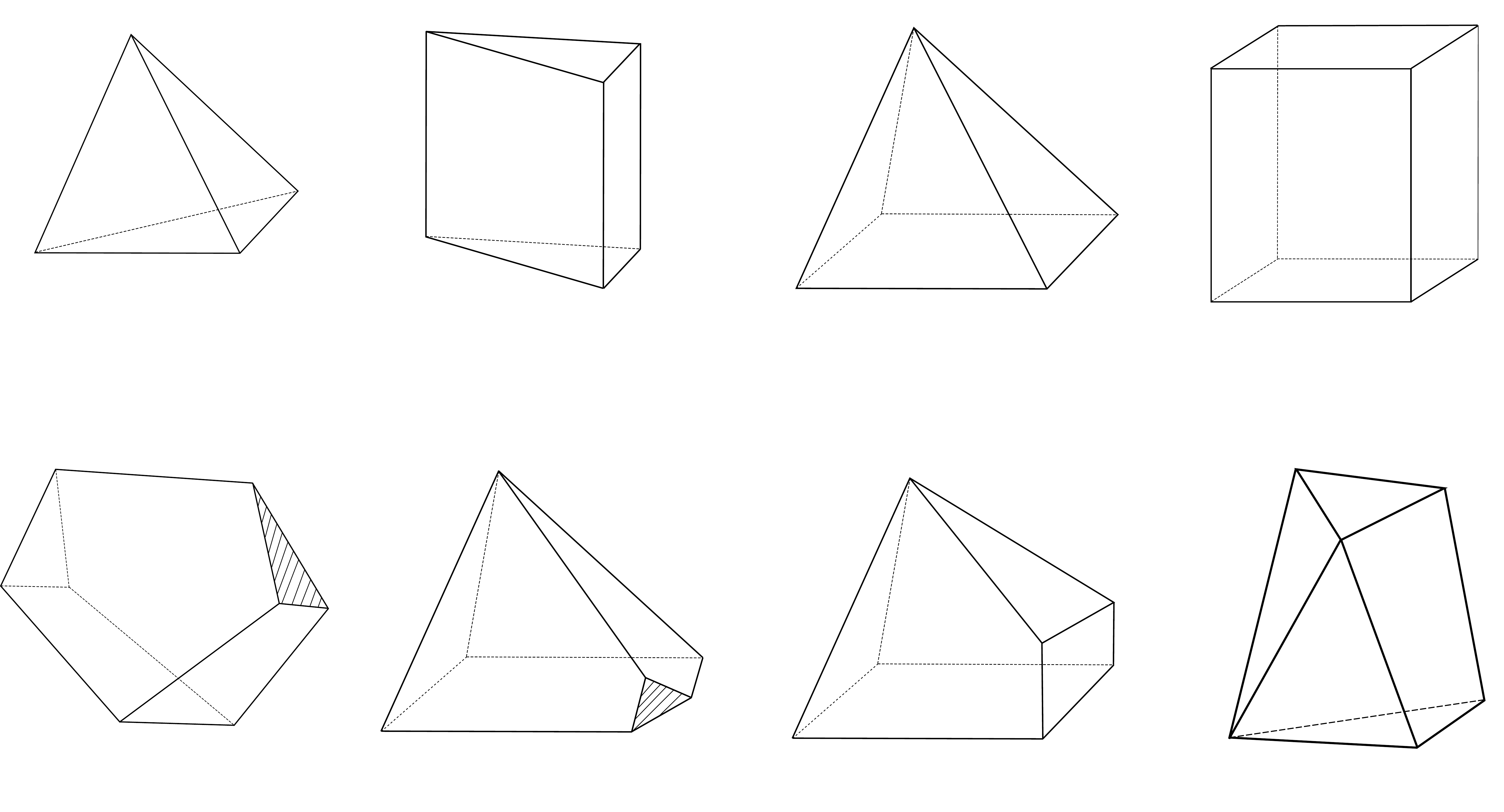
    \caption{}
\end{figure}

We will check condition \ref{condition2} of lemma \ref{lemma:inequality} for
each case listed above.
\begin{itemize}
\item[a.] The complete diagram of this face is a sub-diagram of type 1, and
$\sigma(F) \ge 3$ by lemma \ref{lemma:weightof3face}.
\item[b] The sub-diagram $S(F,(ACFD, CBEF, ABED))$ is of type 2, and $\sigma(F)
\ge 2$.
\item[c.] Two sub-diagrams of type 3 are
$$S(F, (AEB, BCDE, ADC)) \text{ and } S(F, (ADE,BCDE, ACB)),$$
and $\sigma(F) \ge 2$.
\item[d.] There are three sub-diagrams of type 4, each consists of 4 consecutive
2-faces of the cube, and $\sigma(F) \ge 3\times 1/3 = 1$.
\item[e.] For this case there is a sub-diagram of type 2 corresponding to 2-faces
which intersect the shaded triangle, and $\sigma(F) \ge 2$.
\item[f.] This case is similar to e.
\item[g.] The sub-diagram $S(F, (ABEF, BEDC, CDGA))$ is of type 3, and $\sigma(F)
\ge 1$.
\item[h.] The sub-diagram $S(F, (ACD, ADEB, BEFC))$ is of type 3, and $\sigma(F)
\ge 1$.
\end{itemize}
Thus, proposition \ref{prop:checkcondition} and so theorem
\ref{theorem:general_case} is proved.
\end{proof}

\section{Examples and discussion}
So far we consider relatively hyperbolic Coxeter groups of type $HM$ with maximal flats of codimension 1. Naturally we can ask if a Coxeter group $W$ of type $HM$ is relatively hyperbolic relative to its maximal affine subgroups, then must it have a special subgroup which is affine of codimension 1? The answer for this question is no. For $n = 4$ there is an example in \cite{DJL} of 4-dimensional right-angled Coxeter group which is relatively hyperbolic and has maximal abelian subgroup of dimension 2. Here in this section we present an example of relatively hyperbolic Coxeter
group with fundamental domain a 12-cube with maximal flats of dimension 6.The idea is that from the diagram for the
right-angled cube consisting of $12$ vertical edges each labeled $\infty$ we add
edges between some vertices so that the link $L$ of new Coxeter group is the
same type as for the right-angled cube.

In the following example there are 8 vertical edges labeled $\infty$ dividing
into 4 groups $A_1, A_2, B_1, B_2$, each of $A_1$ and $A_2$contains 1 edge, $B_1$ and $B_2$ contain 3 edges. The upper vertex of
$A_1$ is connected to all three upper vertices of $B_1$, the lower vertice of
$A_1$ is connected to all three lower vertices of $B_2$. The upper vertice of
$A_2$ is connected to all three upper vertices of $B_2$, and the lower vertice
of $A_2$ is connected to all three lower vertices of $B_1$.


Using criteria from \cite{Caprace} we can check that this group is relatively hyperbolic Coxeter group with a maximal flat of codimension 6  and a maximal flat of codimension 2 (sub-cubes generated by $A_1 \cup A_2$ or $B_1 \cup B_2$)

It is not known whether the above group acts on hyperbolic space or not. Even in the non-relative case there are not many examples of (word) hyperbolic Coxeter groups which are not actual hyperbolic groups (i.e do not act on hyperbolic space). One class of examples was introduced recently by R. Greene in \cite{Greene}. His technique can be applied to produce similar example for relatively hyperbolic Coxeter groups. Here is a rough idea how to produce such an example. Let $X$ be an octahedron, $Y = X \times I$ be the product of $X$ crossed an interval. If we cone off the two copies $X_1, X_2$ of $X$ in $Y$ we get a 3-sphere. We will triangulate this sphere so that one octahedron, $X_1$, is remained (this will be the nerve of a codimension 1 flat) and there is no other empty square except for the ones in $X_1$. For the cone on $X_2$, apply Przytycki-Swiatkowski \cite{PS} subdivision process. The triangles in $X_2$ are subdivided and there are new vertices in $X_2$, we connect the vertices in $X_1$ with appropriate new vertices in $X_2$ to have a triangulation of 3-sphere. You can check that this triangulation doesn't have any empty squares rather than the ones in $X_1$. And so this is the nerve of a relatively hyperbolic right-angled Coxeter group of type $HM$ with one codimension 1 flat. The fundamental domain of this Coxeter group is not embedded into 4 dimensional hyperbolic space by Greene's argument \cite{Greene}.

\end{document}